\documentclass[12pt,reqno]{amsart}
\usepackage{hyperref}
\usepackage{geometry}
\usepackage{enumitem}
\usepackage{amsmath}
\usepackage{amssymb}
\usepackage{amsthm}
\usepackage{amsrefs}
\usepackage{amsfonts}
\usepackage[dvipsnames]{xcolor}
\usepackage{mathtools}
\usepackage{stackengine}
\usepackage{verbatim}
\usepackage{tikz-cd} 
\usepackage{float}

\makeatletter
\@namedef{subjclassname@2020}{%
  \textup{2020} MSC}
\makeatother

\hypersetup{
	colorlinks   = true, 
	urlcolor     = blue, 
	linkcolor    = purple, 
	citecolor   = blue 
}

\def\XXint#1#2#3{{\setbox0=\hbox{$#1{#2#3}{\int}$ }
\vcenter{\hbox{$#2#3$ }}\kern-.6\wd0}}

\makeatletter
\newcommand*{\rom}[1]{\expandafter\@slowromancap\romannumeral #1@}

\newcommand{\ind}{\protect\raisebox{2pt}{$\chi$}}

\newcommand{\prim}{\text{prim}}

\newcommand{\SL}{\mathrm{SL}}

\newcommand{\GL}{\mathrm{GL}}
\newcommand{\M}{\mathrm{M}}

\newcommand{\X}{\mathcal{X}}

\newcommand{\R}{\mathbb{R}}

\newcommand{\e}{\varepsilon}

\newcommand{\cF}{\mathcal{F}}

\newcommand{\Z}{\mathbb{Z}}

\newcommand{\Sphere}{\mathbb{S}}

\newcommand{\N}{\mathbb{N}}

\newcommand{\bthm}{\begin{thm}}
\newcommand{\ethm}{\end{thm}}
\newcommand{\bproof}{\begin{proof}}
\newcommand{\eproof}{\end{proof}}
\newcommand{\blem}{\begin{lem}}
\newcommand{\elem}{\end{lem}}
\newcommand{\brem}{\begin{rem}}
\newcommand{\erem}{\end{rem}}
\newcommand{\eeqn}{\end{equation}}
\newcommand{\eeqnn}{\end{equation*}}
\newcommand{\beqn}{\begin{equation}}
\newcommand{\beqnn}{\begin{equation*}}
\newcommand{\eprop}{\end{prop}}
\newcommand{\eexm}{\end{exm}}
\newcommand{\enexm}{\end{nexm}}
\newcommand{\ecor}{\end{cor}}
\newcommand{\bcor}{\begin{cor}}
\newcommand{\bexm}{\begin{exm}}
\newcommand{\bnexm}{\begin{nexm}}
\newcommand{\bprop}{\begin{prop}}
\newcommand{\bdefn}{\begin{defn}}
\newcommand{\edefn}{\end{defn}}
\newcommand{\benum}{\begin{enumerate}}
\newcommand{\eenum}{\end{enumerate}}

\newcommand{\cN}{\mathcal{N}}

\newcommand{\Mat}{\M_{m \times n}(\R)}

\newcommand{\supp}{\text{supp}}

\newcommand{\inj}{\operatorname{inj}}
\newcommand{\Lip}{\text{Lip}}

\title{The Doeblin-Lenstra conjecture: effective results and central limit theorems}

\begin{document}
\theoremstyle{plain}
\newtheorem{thm}{Theorem}[section]
\newtheorem{lem}[thm]{Lemma}
\newtheorem{prop}[thm]{Proposition}
\newtheorem{cor}[thm]{Corollary}
\newtheorem{question}{Question}
\newtheorem{con}{Conjecture}
\theoremstyle{definition}
\newtheorem{defn}[thm]{Definition}
\newtheorem{exm}[thm]{Example}
\newtheorem{nexm}[thm]{Non Example}
\newtheorem{prob}[thm]{Problem}

\theoremstyle{remark}
\newtheorem{rem}[thm]{Remark}

\author{Gaurav Aggarwal}
\address{\textbf{Gaurav Aggarwal} \\
School of Mathematics,
Tata Institute of Fundamental Research, Mumbai, India 400005}
\email{gaurav@math.tifr.res.in}

\author{Anish Ghosh}
\address{\textbf{Anish Ghosh} \\
School of Mathematics,
Tata Institute of Fundamental Research, Mumbai, India 400005}
\email{ghosh@math.tifr.res.in}

\date{}

\thanks{G.\ A.\ was an Infosys fellow during the writing of the paper. A.\ G.\ gratefully acknowledges support from a grant from the Infosys foundation to the Infosys Chandrasekharan Random Geometry Centre. G. \ A.\ and  A.\ G.\ gratefully acknowledge a grant from the Department of Atomic Energy, Government of India, under project $12-R\&D-TFR-5.01-0500$. This work was supported by a Royal Society International Exchanges Grant. Part of the work was done when G.\ A.\ was visiting the University of Bristol on a visit funded by this grant. The support of the grant and the hospitality of the University of Bristol are gratefully acknowledged. }



\begin{abstract}  

We establish effective convergence rates in the Doeblin--Lenstra law, describing the limiting distribution of approximation coefficients arising from continued fraction convergents of a typical real number. More generally, we prove quantitative versions of the Doeblin--Lenstra law for best approximates in higher dimensions, as well as for points sampled from fractal measures on the real line, including the middle-third Cantor measure.

Our method reduces the problem to proving effective convergence of Birkhoff averages for diagonal flows on spaces of unimodular lattices. A key step is to show that, despite the discontinuity of the observable of interest, it satisfies the regularity conditions on average required for effective ergodic theorems. For the fractal setting, we establish effective multi-equidistribution properties of self-similar measures under diagonal flow, extending recent work on single equidistribution by B\'enard, He and Zhang. As a consequence, we also obtain central limit theorems for these Diophantine statistics in both classical and fractal settings.
\end{abstract}

\maketitle

\tableofcontents
\section{Introduction}

Let $\theta \in \mathbb{R}$ be irrational, with continued fraction expansion
\[
\theta = a_0 + \frac{1}{a_1 + \frac{1}{a_2 + \frac{1}{a_3 + \cdots}},
}
\]
where $a_0 \in \mathbb{Z}$ and $a_j \in \mathbb{N}$ for $j \ge 1$. The $k$-th convergent of $\theta$ is
\[
\frac{p_k}{q_k} =  a_0+ \frac{1}{a_1 + \frac{1}{a_2 + \frac{1}{ \ddots + \frac{1}{a_k}}}},
\]
where $p_k$ and $q_k$ are coprime integers. The quantity
\[
q_k \big| \theta q_k - p_k \big|
\]
measures the quality of approximation of $\theta$ by its $k$-th convergent.

\subsection{The Doeblin-Lenstra Law}

In 1938, Wolfgang Doeblin outlined a remarkable limiting law for these approximation coefficients. He showed that for Lebesgue-almost every $\theta \in (0,1)$,
\[
\lim_{N \to \infty} \frac{1}{N} \sum_{k=1}^N F\bigl( q_k |\theta q_k - p_k| \bigr) = \int_0^1 F(z)\, d\nu(z)
\]
for any bounded continuous function $F : [0,1] \to \mathbb{R}$, where $\nu$ is the probability measure on $[0,1]$ with density
\[
\frac{d\nu}{dz} = 
\begin{cases}
\displaystyle \frac{1}{\ln 2}, & 0 \le z \le 1/2, \\[4pt]
\displaystyle \frac{1 - 1/z}{\ln 2}, & 1/2 < z \le 1.
\end{cases}
\]

Doeblin provided a sketch of the proof of this result, and it is difficult to reconstitute a complete proof from his paper. His work remained largely unnoticed until Hendrik Lenstra independently conjectured the same result decades later. The law is now known as the \emph{Doeblin-Lenstra law}. A full proof was provided in 1983 by Bosma, Jager, and Wiedijk~\cite{BJW}, using ergodic theory applied to the natural extension of the Gauss map.

Although the qualitative behavior is well understood, effective convergence rates for this law had remained out of reach. In this paper, we establish the first quantitative convergence rates:

\begin{thm}
\label{thm: intro 1}
Let $F : \mathbb{R} \to \mathbb{R}$ be differentiable with bounded derivative. Then for any $\varepsilon > 0$, for Lebesgue-almost every $\theta$,
\[
\frac{1}{N} \sum_{k=1}^N F\bigl( q_k |\theta q_k - p_k| \bigr) = \int_0^1 F(z)\, d\nu(z) + O\bigl( N^{-1/2} \log^{3/2 + \varepsilon} N \bigr).
\]
\end{thm}

A central limit theorem will also be established; we defer its precise formulation to Section~\ref{sec: main result diophantine}. In this context, we note the interesting paper \cite{Knuth} of Donald Knuth, in which a related conjecture of Lenstra is settled. 

\subsection{Mahler’s Problem and Fractals}

In 1984, Mahler posed the fundamental question: What are the typical Diophantine properties of points sampled from sets like the middle-third Cantor set? His query led to a rich field of research. Recent advances include results on badly approximable numbers in fractals~\cites{SimmonsWeiss, AG24Random}, Khintchine-type theorems~\cites{KhalilLuethi, dattajana24, benard24}, and singular vectors on fractals~\cites{aggarwal2025, Khalilsing, AG24singular}.

In this direction, a natural extension of the Doeblin-Lenstra law is to ask whether it, and its quantitative form, hold with respect to measures supported on fractals. This was answered qualitatively by Shapira and Weiss~\cite{SW22}, and more generally by the present authors in~\cite{AG24Levy}, for a broad class of self-similar sets. However, the question of effective convergence rates on fractals remained open. We resolve this by proving:

\begin{thm}
\label{thm: intro 2}
Let $\mu$ be a non-atomic self-similar measure on $\mathbb{R}$ generated by similarities with a common contraction ratio (e.g., the Hausdorff measure on the middle-third Cantor set). Then for any differentiable $F : \mathbb{R} \to \mathbb{R}$ with bounded derivative and any $\varepsilon > 0$, for $\mu$-almost every $\theta$,
\[
\frac{1}{N} \sum_{k=1}^N F\bigl( q_k |\theta q_k - p_k| \bigr) = \int_0^1 F(z)\, d\nu(z) + O\bigl( N^{-1/2} \log^{3/2 + \varepsilon} N \bigr).
\]
\end{thm}

\subsection{Higher-Dimensional Analogues}

In dimension one, the partial quotients (or convergents) of a continued fraction expansion correspond precisely to best approximations of the real number. This observation provides a natural framework for generalizing the Doeblin-Lenstra law: in higher dimensions, one studies best approximations and the associated error terms. Specifically, for $\theta \in M_{m \times n}(\mathbb{R})$, the analogue of the quantity $q_k |\theta q_k - p_k|$ is
\[
\| q \|_{\R^n}^n \| p + \theta q \|_{\R^m}^m
\]
where $(p, q) \in \mathbb{Z}^m \times (\mathbb{Z}^n \setminus \{0\})$ is a best approximation to $\theta$.

Formally we fix norms $\|\cdot\|_{\R^m}$ on $\R^m$ and $\|\cdot\|_{\R^n}$ on $\R^n$. Then $(p,q)$ is called a best approximation of $\theta \in \Mat$ if there is no other $(p', q') \in \mathbb{Z}^m \times (\mathbb{Z}^n \setminus \{0\})$, other than $(\pm p, \pm q)$, satisfying
\[
\| p' + \theta q' \|_{\mathbb{R}^m} \le \| p + \theta q \|_{\mathbb{R}^m} \quad \text{and} \quad \| q' \|_{\mathbb{R}^n} \le \| q \|_{\mathbb{R}^n}.
\]

It was shown in~\cite{SW22} (for $n=1$) and~\cite{AG24Levy} (general $m,n$) that there exists a measure $\nu_{m,n}$ (depending on chosen norms $\|\cdot\|_{\R^m}$ and $\|\cdot\|_{\R^n}$) such that for Lebesgue-almost every $\theta$,
\[
\lim_{N \to \infty} \frac{1}{N} \sum_{k=1}^N F \bigl( \| q_k \|^n \| p_k + \theta q_k \|^m \bigr) = \int_0^1 F(z)\, d\nu_{m,n}(z)
\]
for any bounded continuous $F$. In this paper, we prove the first quantitative refinement of this convergence:

\begin{thm}
\label{thm: intro 3}
Let $m,n \ge 1$ and let $\nu_{m,n}$ be as above. Let $F : \mathbb{R} \to \mathbb{R}$ be differentiable with bounded derivative. Then for any $\varepsilon > 0$, for Lebesgue-almost every $\theta \in M_{m \times n}(\mathbb{R})$,
\[
\frac{1}{N} \sum_{k=1}^N F \bigl( \| q_k \|_{\R^n}^n \| p_k + \theta q_k \|_{\R^m}^m \bigr)
= \int_\R F(z)\, d\nu_{m,n}(z) + O\bigl( N^{-1/2} \log^{3/2 + \varepsilon} N \bigr).
\]
\end{thm}

A central limit theorem in this higher-dimensional setting is also obtained; its statement appears in Section~\ref{sec: main result diophantine}.

\subsection{Strategy of Proof}

The key idea underlying our approach is to observe that the Doeblin-Lenstra law, modulo the L\'evy-Khintchine theorem, can be recast as the convergence of Birkhoff averages for an ergodic transformation—specifically, a diagonal flow-on the space of unimodular lattices $\SL_{m+n}(\mathbb{R})/\SL_{m+n}(\mathbb{Z})$. This perspective has a crucial advantage over more recent approaches that rely on the study of cross-sections in the space of lattices, as developed in~\cites{aggarwalghosh2024joint, AG24Levy, SW22, CC19}, which have been effective in analyzing qualitative properties of approximates but do not readily yield quantitative rates of convergence.

An effective version of the L\'evy-Khintchine theorem was established by the authors in~\cite{AG25}. Consequently, to derive quantitative forms of the Doeblin-Lenstra law, it suffices to establish effective convergence of Birkhoff averages. However, the observable function involved in our setting is not continuous, and its measurability is not immediate. To overcome this difficulty, we employ a criterion from~\cite{AG25}, which asserts that effective convergence of Birkhoff averages—and even a central limit theorem—holds for functions that are approximately invariant under small perturbations, in an average sense.

In the present work, we verify that this regularity condition applies to the function of interest. More precisely, we show that for a small perturbation $g$, the pointwise difference between $f(\cdot)$ and $f(g\cdot)$ is controlled by an unbounded function $\tau_g$ whose average over the space of unimodular lattices is small when $d(g,e)$ is small.

A further challenge lies in extending our results to the fractal setting. A crucial input for proving both effective Birkhoff convergence and the central limit theorem is a quantitative equidistribution property of the underlying measure. More specifically, the measure on $\Mat$ must satisfy what we refer to as \emph{Effective Multi-Equidistribution under the Identity} (Condition~(EMEI)) under the diagonal flow. While effective \emph{single}-equidistribution has received considerable attention recently due to its role in metric Diophantine approximation on fractals (see~\cites{benard24, KLW, KhalilLuethi, dattajana24}), our setting requires a significantly stronger property.

In this paper, we build on and extend the results of B\'enard-He-Zhang~\cite{benard24} by proving that fractal measures on the real line, generated by similarities with equal contraction ratios, satisfy Condition~(EMEI). Our proof crucially relies on the results of~\cite{benard24}.

\section{Main results}

\subsection{The Condition (EMEI) for fractal measures}
\label{sec: Effective Multi-equidistribution of fractal measure}

 Let $G= \SL_{m+n}(\R)$ and $\Gamma = \SL_{m+n}(\Z)$. Let $m_G$ equal Haar measure on $G$ so that the fundamental domain of the $\Gamma$ action on $G$ has measure equal to $1$. Define $\X= G/\Gamma$, then can be identified with the space of all unimodular lattices in $\R^{m+n}$ via the identification 
    $$
    A\Gamma \mapsto A\Z^{m+n}.
    $$
    Let $\mu_\X$ denote the unique $G$-invariant probability measure on $\X$. \\
    
    For $\theta \in \Mat$ and $t \in \R$, let us define $ u(\theta), a_t \in G$ as
    \begin{align}
        u(\theta) = \begin{pmatrix}
            I_m & \theta \\ & I_n
        \end{pmatrix}, \quad  a_t = \begin{pmatrix}
            e^{\frac{n}{m}t} I_m \\ & e^{-t} I_n
        \end{pmatrix}.
    \end{align}

We recall the property (EMEI) introduced in our paper \cite{AG25}.    
\begin{defn}
    A probability measure $\mu$ on $\Mat$ will be called to satisfy Condition (EMEI) (short for \emph{Effective Multi-equidistribution for Identity coset under diagonal flow $a_t$}) if it satisfies the following properties:
    \begin{itemize}
        \item $\mu$ is compactly supported,
        \item For all $F_0 \in C^{\infty}(\Mat )$, $F_1, \ldots, F_r \in C_c^{\infty}(\X)$ and $t_1, \ldots,t_r >0 $, we have 
    \begin{align}
         \int_{\Mat} F_0(\theta) \left( \prod_{i=1}^r F_i(a_{t_i} u(\theta) \Gamma)\right) \, d\mu(\theta) &=  \mu(F_0) \mu_\X( F_1) \cdots \mu_\X(F_r) \nonumber\\
         & \quad + {O}_{ r} \left(e^{-\delta D(t_1, \ldots, t_r)} \|F_0\|_{C^k}  \prod_{i=1}^r \|F_i\|_{C_k} \right),  \label{eq: mix hom identity}
    \end{align}
    where $D(t_1, \ldots, t_r) = \min\{t_i, |t_i - t_j|: 1\leq i ,j \leq r, i \neq j\}$.
    \end{itemize}
\end{defn}

\begin{rem}
\label{rem: class of measures satifying EMEI}
    As already noted in \cite{AG25}, it is known that if $\mu$ is equal to the restriction of the Lebesgue measure to $\M_{m \times n}([0,1])$, then $\mu$ satisfies the condition \textnormal{(EMEI)}, see \cite[Cor.~3.5]{KM3}. Now, we will prove that if $\mu$ is the Bernoulli measure in the limit set of the family of invertible affine maps with constant ratio and without fixed point in $\R$, then $\mu$ satisfies the condition (EMEI). In particular, the condition (EMEI) holds for the normalized restriction of the $\log 2/\log 3$-dimensional Hausdorff measure on the middle third Cantor set.
\end{rem}

Assume that $\{\kappa_e: e \in E\}$ is a family of invertible affine similarities on $\Mat$ of the form
\begin{align}
\label{eq: def kappa _ e}
    \kappa_e(\theta) := \rho( O_e \theta O_e' + w_e),
\end{align}
for some fixed $\rho \in (0,1)$, where $O_e$ (resp. $O_e'$) varies in a fixed compact subset $K_m$ (resp. $K_n$) of $\GL_m(\R)$ (resp. $\GL_n(\R)$). Assume that the set $E$ is compact and that  $\nu$ is a probability measure on $E$.

Let $\hat{\kappa}: E^\N \rightarrow \Mat $ denote the map 
\begin{align}
\label{eq: def hat kappa}
    \hat{\kappa}(\underline{e}) = \lim_{k \rightarrow \infty} \kappa_{e_1} \circ \cdots \circ \kappa_{e_k}(0),
\end{align}
for $\underline{e}=(e_1, e_2, \ldots) \in E^\N$. It is easy to see that the limit exists and, in fact, $\hat{\kappa}$ is continuous.  

Let $\mu$ denote the measure on $\Mat$ obtained by pushing forward  $\nu^{\otimes \N}$ under the map $\hat{\kappa}$.  \vspace{0.2in}

The first main theorem of this paper is the following.
\begin{thm}
\label{thm: Effective Multi-equidistribution of fractal measure}
    Fix a Bernoulli measure $\mu$ on $\Mat$ as above. Assume that there exists $\delta_\mu>0$ and $k \geq 1$ such that for all $F \in C_c^\infty(\X)$, and $x \in \X$, we have
\begin{align}
    \label{eq: equidistribution}
    \int_{\Mat} F(a_t u(\theta) x) \, d\mu(\theta) = \mu_\X(F) + O(\inj(x)^{-\beta} e^{- \delta_\mu t} \|F\|_{C^k}).
\end{align}
Then $\mu$ satisfies the condition (EMEI). In fact, the following stronger property holds. There exists a $\delta > 0$ such that for $x \in \X$, $F_0 \in C^{\infty}(\Mat )$, $F_1, \ldots, F_r \in C_c^{\infty}(\X)$ and $t_1, \ldots,t_r >0 $, we have 
    \begin{align}
    \label{eq: mix hom fractal}
         \int_{\Mat} F_0(\theta) \left( \prod_{i=1}^r F_i(a_{t_i} u(\theta) x)\right) \, d\mu(\theta) &= \mu(F_0) \mu_\X( F_1) \cdots \mu_\X(F_r) \nonumber\\
         & \quad + {O}_{x, r} \left(e^{-\delta D(t_1, \ldots, t_r)} \|F_0\|_{C^k}  \prod_{i=1}^r \|F_i\|_{C_k} \right),
    \end{align}
    where $D(t_1, \ldots, t_r) = \min\{t_i, |t_i - t_j|: 1\leq i ,j \leq r, i \neq j\}$.
\end{thm} \vspace{0.2in}

The condition \eqref{eq: equidistribution} for the self-similar probability measures in the real line has been proved in \cite[Thm.~B]{benard24}.
\begin{thm}[{\cite[Thm.~B]{benard24}}]
\label{thm: benardhezhang24}
 Let $\mu$ be a non-atomic self-similar probability measure on $\R$. There exists a constant $c := c(\mu) > 0$ such that for all $t > 1, x \in X, f \in B^{\infty}_{\infty, 1}(\X)$ we have
 $$\int_{\R} f(a(t)u(s)x) d\mu(s) = \int_{\X}f d\mu_{\X} + O(inj(x)^{-1}S_{\infty, 1}(f)t^{-c}), $$
 where the implicit constant in $O(\cdot)$ only depends on $x$ and $\sigma$. 
\end{thm} 
\begin{rem}
    For precise definitions of the function space $B^{\infty}_{\infty,1}(\X)$ and the Sobolev-type norms $S_{\infty,1}$, we refer the reader to~\cite{benard24}. For the purposes of this paper, it suffices to note that $C_c^\infty(\X) \subset B^{\infty}_{\infty,1}(\X)$ and that for all $f \in C_c^\infty(\X)$ and $k \geq 1$, we have
    \[
    S_{\infty,1}(f) \leq \|f\|_{C^k}.
    \]
\end{rem}

\vspace{0.2in}

Combining Theorem \ref{thm: Effective Multi-equidistribution of fractal measure} with {\cite[Thm.~B]{benard24}}, we get the following corollary.

\begin{cor}
\label{cor: Canter satisfies EMEI}
    Let $\mu$ be a non-atomic self-similar measure on $\mathbb{R}$ generated by similarities with a common contraction ratio. Then $\mu$ satisfies the condition (EMEI).
\end{cor}

In particular, Corollary \ref{cor: Canter satisfies EMEI} implies that the normalised restriction of $\log 2/\log 3$-dimensional Hausdorff measure on middle third Cantor set satisfies condition (EMEI).

\vspace{0.2in}

\subsection{Diophantine approximation}
\label{sec: main result diophantine}

Fix norms on $\R^m$ and $\R^n$, and denote both of them by $\|.\|$. Fix $\theta \in \Mat$, and let $(p_k,q_k)_{k \in \Z}$ denote the sequence of best approximates of $\theta$. In this paper, we not only prove effective equidistribution of $\|q_k\|^n\|p_k+ \theta q_k\|^m$, but more generally study the effective joint equidistribution of 
$$
\left( \|p_k+\theta q_k\|^m \|q_k\|^n, \frac{p_k+\theta q_k}{\|p_k+\theta q_k\|} , \frac{q_k}{\|q_k\|}  \right),
$$
that is, not only the quality of approximation, but also the direction from which approximation appears. This, in particular,  implies the effective equidistribution of the quantity
$$
\|q_k\|^{n/m} (p_k+\theta q_k).
$$

Formally, we prove the following result.

\begin{thm}
    \label{thm:DLCEMEI}
    Fix $\mu$ on $\Mat$ satisfying condition (EMEI). Fix a function $F$ on $\R_{\geq 0} \times \Sphere^m \times \Sphere^n$ with bounded first derivative. Then there exists  $\beta > 0$ (depending only on $F$, $m,n$ and choice of norms, and independent of $\mu$) such that the following holds. For any $\varepsilon > 0$ and for $\mu$-almost every $\theta \in \Mat$, we have
    \begin{align}
        \label{eq: thm: main effective levy general}
       \frac{1}{N} \sum_{k=1}^N F\left( \|p_k+\theta q_k\|^m \|q_k\|^n, \frac{p_k+\theta q_k}{\|p_k+\theta q_k\|} , \frac{q_k}{\|q_k\|}  \right) = \beta  + O\left(N^{-1/2} \log^{\frac{3}{2}+\varepsilon} N\right),
    \end{align}
    where $(p_k,q_k)_k$ denote the sequence of best approximates of $\theta$, ordered according to increasing $\|q_k\|$.
\end{thm}
\begin{rem}
    Theorems~\ref{thm: intro 1} and~\ref{thm: intro 3} follow from Theorem~\ref{thm:DLCEMEI} together with the discussion in Remark~\ref{rem: class of measures satifying EMEI}. Similarly, Theorem~\ref{thm: intro 2} follows from Theorem~\ref{thm:DLCEMEI} and Corollary~\ref{cor: Canter satisfies EMEI}.
\end{rem}

\vspace{0.2in}

To prove Theorem~\ref{thm:DLCEMEI}, we begin by defining, for any function $F$ on $\R_{\geq 0} \times \Sphere^m \times \Sphere^n$,
\begin{align}
    \label{eq: def cF}
    \cF_F(\theta,T) := \sum_{\substack{(p,q)\text{ is a best approximate} \\ \|q\| < e^T}} F\left( \|p + \theta q\|^m \|q\|^n,\ \frac{p + \theta q}{\|p + \theta q\|},\ \frac{q}{\|q\|} \right).
\end{align}
Theorem~\ref{thm:DLCEMEI} follows from the result stated below, together with effective equidistribution results from~\cite{AG25}; see Remark~\ref{rem: deriving DLCEMEI} for the derivation.

\begin{thm}
\label{thm: main theorem general}
Let $\mu$ be a measure on $\Mat$ satisfying condition~\textnormal{(EMEI)}, and let $F$ be a function on $\R_{\geq 0} \times \Sphere^m \times \Sphere^n$ with bounded first derivatives. Then there exist constants $\gamma, \sigma > 0$ (depending only on $F$, $m$, $n$, and the choice of norms, and independent of $\mu$) such that:
\begin{itemize}
    \item[(i)] For any $\varepsilon > 0$ and for $\mu$-almost every $\theta \in \Mat$,
    \begin{align}
        \label{eq: thm: main effective DL general}
        \cF_F(\theta, T) = \gamma T + O\left(T^{1/2} \log^{\frac{3}{2} + \varepsilon} T\right).
    \end{align}
    
    \item[(ii)] For every $\xi \in \mathbb{R}$,
    \[
    \mu\left( \left\{ \theta \in \Mat : \frac{\cF_F(\theta, T) - \gamma T}{T^{1/2}} < \xi \right\} \right) \longrightarrow \mathrm{Norm}_{\sigma}(\xi)
    \quad \text{as } T \to \infty,
    \]
    where
    \[
    \mathrm{Norm}_{\sigma}(\xi) := \frac{1}{\sigma \sqrt{2\pi}} \int_{-\infty}^\xi \exp\left(-\frac{1}{2} \left( \frac{x}{\sigma} \right)^2 \right) \, dx.
    \]
\end{itemize}
\end{thm}

\begin{rem}
\label{rem: deriving DLCEMEI}
To deduce Theorem~\ref{thm:DLCEMEI} from Theorem~\ref{thm: main theorem general}, define
\[
\cN(\theta, T) := \#\left\{(p,q) \in \Z^m \times (\Z^n \setminus \{0\}) : (p,q) \text{ is a best approximate and } \|q\| < e^T \right\}.
\]
Fix a measure $\mu$ on $\Mat$ satisfying condition~\textnormal{(EMEI)}, and fix $\varepsilon > 0$. Then by~\cite[Thm. 2.3]{AG25}, there exist constants $\gamma_0, \sigma_0 > 0$ (independent of $\mu$) such that for $\mu$-almost every $\theta$,
\begin{align}
    \label{eq:ag25 eq 1}
    \cN(\theta, T) = \gamma_0 T + O\left(T^{1/2} \log^{\frac{3}{2} + \varepsilon} T\right).
\end{align}

Now fix any $\theta \in \Mat$ satisfying both~\eqref{eq: thm: main effective DL general} and~\eqref{eq:ag25 eq 1}. Let $(p_k, q_k)_k$ denote the sequence of best approximates to $\theta$, ordered by increasing $\|q_k\|$. For each $N \in \mathbb{N}$, define $T_N := \log \|q_N\|$, where $(p_N, q_N)$ is the $N$-th best approximate. Then by construction,
\begin{align}
    \label{eq: ag25 a 1}
    \cN(\theta, T_N) &= N, \\
    \label{eq: ag25 a 2}
    \cF_F(\theta, T_N) &= \sum_{k=1}^N F\left( \|p_k + \theta q_k\|^m \|q_k\|^n,\ \frac{p_k + \theta q_k}{\|p_k + \theta q_k\|},\ \frac{q_k}{\|q_k\|} \right).
\end{align}
Equations~\eqref{eq:ag25 eq 1} and~\eqref{eq: ag25 a 1} yield the bounds
\begin{align}
    \label{eq:ag25 eq 4}
    \frac{1}{2} \gamma_0 T_N \leq N \leq 2 \gamma_0 T_N,
\end{align}
for all sufficiently large $N$.

Putting everything together, we obtain:
\begin{align*}
    &\frac{1}{N} \sum_{k=1}^N F\left( \|p_k + \theta q_k\|^m \|q_k\|^n,\ \frac{p_k + \theta q_k}{\|p_k + \theta q_k\|},\ \frac{q_k}{\|q_k\|} \right) \\
    &= \frac{\cF_F(\theta, T_N)}{\cN(\theta, T_N)} && \text{(by \eqref{eq: ag25 a 1}, \eqref{eq: ag25 a 2})} \\
    &= \frac{\gamma T_N + O\left(T_N^{1/2} \log^{\frac{3}{2} + \varepsilon} T_N\right)}{\gamma_0 T_N + O\left(T_N^{1/2} \log^{\frac{3}{2} + \varepsilon} T_N\right)} && \text{(by \eqref{eq: thm: main effective DL general}, \eqref{eq:ag25 eq 1})} \\
    &= \frac{\gamma}{\gamma_0} + O\left(T_N^{-1/2} \log^{\frac{3}{2} + \varepsilon} T_N\right) \\
    &= \beta + O\left(N^{-1/2} \log^{\frac{3}{2} + \varepsilon} N\right) && \text{(using \eqref{eq:ag25 eq 4}),}
\end{align*}
where $\beta := \gamma / \gamma_0$. This completes the derivation, showing how Theorem~\ref{thm: main theorem general} implies Theorem~\ref{thm:DLCEMEI}. Moreover, it also illustrates that Theorem~\ref{thm: main theorem general} can be viewed as a central limit theorem analogue for the Doeblin–Lenstra law.
\end{rem}


\section{Proof of Theorem \ref{thm: Effective Multi-equidistribution of fractal measure}}

This section aims to prove Theorem \ref{thm: Effective Multi-equidistribution of fractal measure}. We will assume the notation as in Section \ref{sec: Effective Multi-equidistribution of fractal measure}. That is, we fix a family $\{\kappa_e: e \in E\}$ of invertible affine similarities on $\Mat$, given by \eqref{eq: def kappa _ e}. We assume $E$ is a compact set with probability measure $\nu$, and $\mu$ equals the pushforward of $\nu^{\otimes \N}$ under $\hat{\kappa}$, defined as in \eqref{eq: def hat kappa}. We assume that $\mu$ satisfies the conditions of Theorem \ref{thm: Effective Multi-equidistribution of fractal measure}. 

We will denote by $H$ the group
$$
H:= \left\{ \begin{pmatrix}
    A & B \\ & C
\end{pmatrix} \in \SL_d(\R): A \in \GL_m (\R), C \in \GL_n(\R) \text{ and } B \in \Mat  \right\},
$$
and define
\begin{align}
    \label{eq: def hat a}
    \hat{a} =\begin{pmatrix}
    \rho^{\frac{-m}{m+n}} I_m   \\ &\rho^{\frac{n}{m+n}}  I_n
\end{pmatrix}.
\end{align}
We assume that $E \subset H$ by mapping each algebraic similarity $\theta \mapsto \kappa_e(\theta) := \rho( O_e \theta O_e' + w_e)$ to the following element of $G$,
$$
e:= \begin{pmatrix}
    O_e^{-1} \\ & O_e'
\end{pmatrix} \begin{pmatrix}
    I_m & w_e \\ & I_n
\end{pmatrix} \begin{pmatrix}
    \rho^{\frac{-m}{m+n}} I_m   \\ &\rho^{\frac{n}{m+n}}  I_n
\end{pmatrix}.
$$
Define the map $\Xi: H \rightarrow G$ as 
$$
 \Xi\begin{pmatrix}
    A & B \\ & C
\end{pmatrix} = \begin{pmatrix}
    \frac{1}{|\det(A)|^{1/m}} A \\ & \frac{1}{|\det(C)|^{1/n}} C
\end{pmatrix}^{-1}.
$$
Then the following identities hold for all $\underline{e}=(e_1, e_2, \ldots) \in E^\N$:
\begin{align}
  \label{eq: Xi imp}  \Xi(e_l \ldots e_1) &= \left(\begin{pmatrix}
        O_{e_l} \\ & O_{e_l}'
    \end{pmatrix} \cdots \begin{pmatrix}
        O_{e_1} \\ & O_{e_1}' 
    \end{pmatrix}  \right)^{-1}, \\
     \hat{a}^l u(\kappa_{e_1} \circ \cdots \kappa_{e_l}(\theta)) &= \Xi(e_l \ldots e_1) u(\theta ) e_l \cdots e_1, \label{eq: important conjugation for fractal 1}\\
     \hat{a}^l u(\hat{\kappa}({\underline{e}}) ) &= \Xi(e_l \ldots e_1) u(\hat{\kappa}(\tau^l({\underline{e}}))) e_l \cdots e_1, \label{eq: important conjugation for fractal 2}
\end{align}
where $\tau : E^\N \rightarrow E^\N$ denotes the shift map, and

For $l \geq 1$, define the $l$-fold convolution of $\nu$ on $H$ as
\begin{align}
    \nu_l := \nu * \cdots * \nu.
\end{align}
Define
\begin{align}
    \label{eq: def delta 1}
    \delta_1= \delta \cdot \log(\rho^{\frac{-m}{m+n}}).
\end{align}
Fix a constant $\alpha_0 > 1$ such that for all $A \in K_m$, $C \in K_n$, $B \in \hat{\kappa}(E^\N)$, and $x \in \X$, we have
\begin{align}
    \label{eq:abc 1}
    \alpha_0^{-1} \inj(x) \leq  \inj\left( 
    \begin{pmatrix}
        A & \\ & C
    \end{pmatrix}
    \begin{pmatrix}
        I_m & B \\ & I_n
    \end{pmatrix} x  
    \right) \leq \alpha_0 \inj(x).
\end{align}
Finally, let $\gamma > 0$ be such that for all $\epsilon > 0$,
\begin{align}
    \mu_\X\left(\left\{ y \in \X : \inj(y) < \epsilon \right\}\right) \ll e^{-\gamma}.
\end{align}

\subsection{Non-Escape of Mass}

The main aim of this subsection is to prove the following lemma, which implies that the random walk on $\X$ induced by the measure $\nu$ does not diverge on average. The following lemma will be used in proof of Lemma \ref{lem: induction step}.
\begin{lem}
    \label{lem: escape of mass}
    Fix $x \in \X$. Then for all $l \geq 1$ and $\e>0$, we have
    $$
   \nu_l \left\{g \in G: \inj(g.x)<\e \right\} \ll \e+ e^{-l},
    $$
    where implied constant depend only on $X$.
\end{lem}
\begin{proof}
   Fix a non-negative smooth function $\eta_0$ on $G$, supported on a small ball around the identity in $G$, such that
   \begin{align*}
    \int_G \eta_0(g)\, dm_G(g) &= 1, \\
    \|\eta_0\|_{C^j} &< \infty.
\end{align*}
   
   Also fix $\alpha_1>0$ such that for all $x \in \X$ and $g \in \supp(\eta_{{0}})$, we have
   \begin{align}
        \label{eq:abc 2}
        \alpha_1^{-1} \inj(x) \leq  \inj\left(g x  \right) \leq \alpha_1 \inj(x).
    \end{align}
    Then by definition of $\nu_l$, we have
    \begin{align}
    \label{eq: eq 1}
        \nu_l \left\{g \in G: \inj(g.x)<\e \right\} = \nu^{\otimes l} \left\{ (e_1, \ldots, e_l):  \inj(e_l \cdots e_1.x)<\e \right\} .
    \end{align}
    Now by definition of $\alpha_0$, we have $\inj(e_l \cdots e_1.x)<\e$ implies that $$\inj(\Xi(e_l \ldots e_1) u(\hat{\kappa}(e_{l+1}, e_{l+2}, \ldots)) e_l \cdots e_1.x) < \alpha_0 \e,$$ for all $ (e_{l+1}, e_{l+2}, \ldots) \in E^\N $. Therefore \eqref{eq: eq 1} implies that
    \begin{align}
    \nonumber
        &\nu_l \left\{g \in G: \inj(g.x)<\e \right\} \\
        &\leq  \nu^{\otimes l} \otimes \nu^{\otimes \N} \left\{ \underline{e}=( e_1, e_2, \ldots, ) \in  E^\N: \inj(\Xi(e_l \ldots e_1) u(\hat{\kappa}(e_{l+1}, e_{l+2}, \ldots)) e_l \cdots e_1.x) < \alpha_0 \e \right\} \nonumber\\
        &= \nu^{\otimes \N} \left\{ \underline{e}=( e_1, \ldots, e_l, e_{l+1}, \ldots) \in  E^\N: \inj(\Xi(e_l \ldots e_1) u(\hat{\kappa}(\tau^l(\underline{e}))) e_l \cdots e_1 \cdot x) < \alpha_0 \e \right\} \nonumber\\
        &= \nu^{\otimes \N} \left\{ \underline{e} \in  E^\N: \inj(\hat{a}^l u(\hat{\kappa}(\underline{e}))x ) < \alpha_0 \e \right\},  \quad \text{ using  equation \eqref{eq: important conjugation for fractal 2}}  \nonumber \\
        &= \mu \left\{ \theta \in \Mat: \inj(\hat{a}^l u(\theta)x) < \alpha_0 \e  \right\}, \label{eq: eq 2}
    \end{align}
    where in last equality we used the fact that $\mu $ equals the pushforward of $\nu^{\otimes \N}$ under $ \hat{\kappa}$.

    Now using definition of $\alpha_1$, the condition $\inj(\hat{a}^l u(\theta) x) < \alpha_0 \e$ implies that 
    $$
        \inj(g^{-1} \cdot \hat{a}^l u(\theta)x) < \alpha_0 \alpha_1 \e,
    $$
    for all $g \in \supp(\eta_{0})$. Therefore, equation \eqref{eq: eq 2} implies that
    \begin{align}
        &\nu_l \left\{g \in G: \inj(g.x)<\e \right\} \nonumber\\
        &\leq \int_G \eta_{0}(g)  \mu \left\{ \theta \in \Mat: \inj(g^{-1} \hat{a}^l u(\theta)x) < \alpha_0 \alpha_1 \e  \right\} \, dm_G(g)\nonumber \\
        &= \int_G \int_{\Mat} \eta_{0}(g) \ind_{\{y: \inj(y)< \alpha_0 \alpha_1 \e\}}(g^{-1} \hat{a}^l u(\theta) x) \, dm_G(g) d\mu(\theta) \nonumber\\
        &= \int_{\Mat} (\eta_{0} * \ind_{\{y: \inj(y)< \alpha_0 \alpha_1 \e\}})(\hat{a}^l u(\theta) x) \nonumber \\
        &= 1- \int_{\Mat} (\eta_{0} * \ind_{\{y: \inj(y)> \alpha_0 \alpha_1 \e\}})(\hat{a}^l u(\theta) x) . \label{eq: eq 3}
    \end{align}
    Now note that the function $$ \eta_{0} * \ind_{\{y: \inj(y)> \alpha_0 \alpha_1 \e\}}$$ is a  compactly supported continuous function on $\X$. Therefore \eqref{eq: equidistribution} and \eqref{eq: eq 3} gives that
    \begin{align*}
         &\nu_l \left\{g \in G: \inj(g.x)<\e \right\} \nonumber\\
        &= 1 - \mu_\X(\eta_{0} * \ind_{\{y: \inj(y)> \alpha_1 \alpha_0 \e\}}) + O( \inj(x)^{-\beta} e^{-\delta_1 l} \|\eta_{0} * \ind_{\{y: \inj(y)> \alpha_1 \alpha_0 \e\}}\|_{C^k}) \\
        &= \mu_{\X} (\{y: \inj(y)< \alpha_1 \alpha_0 \e\}) + + O( \inj(x)^{-\beta} e^{-\delta_1 l} \|\eta_{0}\|_{C^l} \| \ind_{\{y: \inj(y)> \alpha_1 \alpha_0 \e\}}\|_{C^0}) \\
        &\ll \e^\gamma +  \inj(x)^{-\beta} e^{-\delta_1 l},
    \end{align*}
  where implied constant depend only on choice of $\eta_0$, and is independent of $\e$. Hence proved.
\end{proof}

\subsection{Induction Step}
We will prove Theorem \ref{thm: Effective Multi-equidistribution of fractal measure} by using induction on $r$. The following lemma will be crucial for the inductive step.
\begin{lem}
\label{lem: induction step}
    There exists a $0<\delta_0< -\log \rho$ such that following holds for all $x \in \X$, $F_1 \in C^\infty(\Mat)$, $F_2 \in C_c^\infty(\X)$ and natural numbers $l<j$, we have
    \begin{align*}
         \int_{\Mat} F_1(\theta) F_2(\hat{a}^j u(\theta) x) \, d\nu(\theta) &= \mu(F_1) \mu_{\X}(F_2) + O_x( e^{-\delta_0 (j-l) } \|F_1\|_{C^0} \|F_2\|_{C^k}  ) \\
         &\quad + O_x(  e^{-\delta_0 l} \|F_1\|_{C^0} \|F_2\|_{C^0} + \rho^l  \|F_1\|_{\Lip} \|F_2\|_{C^0}).
    \end{align*}
\end{lem}
\begin{proof}
    Using fact that $\mu $ equals the pushforward of $\nu^{\otimes \N}$ under $ \hat{\kappa}$, we have
    \begin{align}
       &\int_{\Mat} F_1(\theta) F_2(\hat{a}^j u(\theta) x) \, d\nu(\theta) \nonumber  \\
        &= \int_{E^l} \int_{\Mat} F_1(\kappa_{e_1} \circ \cdots \circ \kappa_{e_l}(\theta)) F_2(\hat{a}^{j} u(\kappa_{e_1} \circ \cdots \circ \kappa_{e_l}(\theta)) x ) \, d\mu(\theta) d\nu^{l}(e_1, \ldots, e_l). \label{eq: 1 eq}
    \end{align}
    Now using equations \eqref{eq: important conjugation for fractal 1} and \eqref{eq: 1 eq}, along with fact that image of $\Xi$ and $\hat{a}$ commutes, we get that
    \begin{align}
        & \int_{\Mat} F_1(\theta) F_2(\hat{a}^j u(\theta) x) \, d\nu(\theta) \nonumber \\
    &= \int_{E^l} \int_{\Mat} F_1(\kappa_{e_1} \circ \cdots \circ \kappa_{e_l}(\theta)) F_2 \circ \Xi(e_l \cdots e_1) (\hat{a}^{j-l} \   u(\theta) \  e_l \cdots e_1 \  x ) \, d\mu(\theta) d\nu^{l}(e_1, \ldots, e_l) \nonumber \\
    &= I_1(\e) + I_2(\e), \label{eq: 2 eq}
    \end{align}
    where $I_1(\e)$ is integral of the function
    $$
    (e_1, \ldots, e_l) \mapsto \int_{\Mat}  F_1( \kappa_{e_1} \circ \ldots \kappa_{e_l}(\theta)) F_2 \circ \Xi(e_l \cdots e_1)   (\hat{a}^{j-l} u(\theta) e_l \cdots e_1 x ) \, d\mu(\theta) 
    $$ 
    over the set $J_1(\e)$, consisting of all $(e_1, \ldots, e_l)$ such that $\inj(e_l \cdots e_1 x)>\e$ and $I_2(\e)$ denotes the integral over the complement $J_2(\e) = \Mat \setminus J_1(\e)$.
    
    We first estimate $I_2(\e)$ by using Lemma \ref{lem: escape of mass}, to get
    \begin{align}
        \label{eq: abcd 1}
        |I_2(\e)| \ll (\e^{\gamma} + e^{-\delta_1 l} \inj(x)^{\beta}) \|F_1\|_{C^0} \|F_2\|_{C^0}. 
    \end{align}

    To estimate $I_1(\e)$, first of all note that for all $(e_1, \ldots, e_j ) \in E^j$ we have 
    $$
    F_1( \kappa_{e_1} \circ \ldots \kappa_{e_l}(\theta)) = F_1(\kappa_{e_1} \circ \ldots \kappa_{e_l}(0)) + O(\|F_1\|_{\Lip} \rho^l ),
    $$
    where implied constant depend only on $\supp(\mu)$. Therefore we have
    \begin{align}
        &\int_{\Mat}  F_1( \kappa_{e_1} \circ \ldots \kappa_{e_l}(\theta)) F_2 \circ \Xi(e_l \cdots e_1)   (\hat{a}^{j-l} u(\theta) e_l \cdots e_1 x ) \, d\mu(\theta) \nonumber \\
        &=  F_1(\kappa_{e_1} \circ \ldots \kappa_{e_l}(0)) \cdot \int_{\Mat} F_2 \circ \Xi(e_l \cdots e_1)   (\hat{a}^{j-l} u(\theta) e_l \cdots e_1 x ) \, d\mu(\theta) + O(\rho^l \|F_1\|_{\Lip} \|F_2\|_{C^0}). \label{eq: 3 eq}
    \end{align}
    Now assume that $\inj(e_l \cdots e_1 x)>\e$, then using equations \eqref{eq: equidistribution} and \eqref{eq: 3 eq}, we have 
    \begin{align}
       &\int_{\Mat}  F_1( \kappa_{e_1} \circ \ldots \kappa_{e_l}(\theta)) F_2 \circ \Xi(e_l \cdots e_1)   (\hat{a}^{j-l} u(\theta) e_l \cdots e_1 x ) \, d\mu(\theta) \nonumber \\
       &= F_1(\kappa_{e_1} \circ \ldots \kappa_{e_l}(0)) \cdot \mu_\X (F_2)   + O(e^{-\delta_1 (j-l)} \|F_2\|_{C^k} \|F_1\|_{C^0} \e^{-\beta}  + \rho^l  \|F_1\|_{\Lip} \|F_2\|_{C^0} ) \nonumber \\
       &= \int_{\Mat} F_1(\kappa_{e_1} \circ \ldots \kappa_{e_l}(\theta)) \, d\mu(\theta)  \cdot \mu_\X (F_2) + O(e^{-\delta_1 (j-l)} \|F_2\|_{C^k} \|F_1\|_{C^0} \e^{-\beta}   + \rho^l  \|F_1\|_{\Lip} \|F_2\|_{C^0} ) \label{eq: 4 eq}
    \end{align}
    Equation \eqref{eq: 4 eq} implies that
    \begin{align}
        I_1(\e) &= \mu_\X (F_2) \cdot \int_{J_1(\e)} \int_{\Mat }F_1(\kappa_{e_1} \circ \ldots \kappa_{e_l}(\theta)) \, d\mu(\theta) \nonumber \\
        & \quad +  O(e^{-\delta_1 (j-l)} \|F_2\|_{C^k} \|F_1\|_{C^0} \e^{-\beta} + \rho^l \|F_1\|_{\Lip} \|F_2\|_{C^0}).  \label{eq: 5 eq}
    \end{align}
    Now by using Lemma \ref{lem: escape of mass}, we have $\mu(\Mat \setminus J_1(\e)) \ll \e^{\gamma} + e^{-\delta_1 l} \inj(x)^{\beta}$. Therefore \eqref{eq: 5 eq} implies that
    \begin{align}
        I_1(\e )&= \mu_\X (F_2) \cdot \int_{E^l} \int_{\Mat }F_1(\kappa_{e_1} \circ \ldots \kappa_{e_l}(\theta)) \, d\mu(\theta) d\nu^l(e_1, \ldots, e_l) \nonumber \\
        &\quad + O( (\e^{\gamma} + e^{-\delta_1 l} \inj(x)^{\beta}) \|F_1\|_{C^0} \|F_2\|_{C^0}) \nonumber \\
        & \quad +  O(e^{-\delta_1 (j-l)} \|F_2\|_{C^k} \|F_1\|_{C^0} \e^{-\beta} + \rho^l \|F_1\|_{\Lip} \|F_2\|_{C^0} ).  \label{eq: 6 eq}
    \end{align}
    Now using fact that 
    $$
    \int_{E^l} \int_{\Mat }F_1(\kappa_{e_1} \circ \ldots \kappa_{e_l}(\theta)) \, d\mu(\theta) = \int_{\Mat }F_1(\theta) \, d\mu(\theta),
    $$
    the equations \eqref{eq: 2 eq}, \eqref{eq: abcd 1} and \eqref{eq: 6 eq} implies that
    \begin{align}
        \int_{\Mat} F_1(\theta) F_2(\hat{a}^j u(\theta) x) \, d\nu(\theta)  &= \mu(F_1) \mu_{\X}(F_2) + O( (\e^{\gamma} + e^{-\delta_1 l} \inj(x)^{\beta}) \|F_1\|_{C^0} \|F_2\|_{C^0}) \nonumber \\
        & + O(e^{-\delta_1 (j-l)} \|F_2\|_{C^k} \|F_1\|_{C^0} \e^{-\beta} + \rho^l \|F_1\|_{\Lip} \|F_2\|_{C^0} ). \label{eq: 7 eq}
    \end{align}
     Putting $\e= e^{-\frac{\delta_1(j-l)}{\beta+ \gamma} }$, and $\delta_0= \min\{ \frac{\gamma \delta_1}{\beta+ \gamma} , \delta_{\mu} \}$ in \eqref{eq: 7 eq}, the lemma follows.
\end{proof}

To apply Lemma \ref{lem: induction step} in the inductive step, we will need to treat functions on $\X$ as functions on $\Mat$. The following lemma bounds the Lipschitz norm of such a function.

\begin{lem}
    \label{lem: C^k norm of restricted function}
   For each $ r\geq 1$, there exists a constant $c_r>0$ such that the following holds. Suppose $F_0 \in C^\infty(\Mat)$ and $F_1, \ldots, F_r \in C_c^\infty(\X)$. Suppose $l_1, \ldots, l_r \in \N$ and $x \in \X$. Then the Lipschitz norm of the function 
   \begin{align}
       \label{eq: aaa 1}
   \theta \mapsto F_0(\theta) F_1(\hat{a}^{l_1} u(\theta)x) \cdots F_r(\hat{a}^{l_r} u(\theta)x),
   \end{align}
  is less than $c_r \rho^{\max\{l_i\}} \|F_0\|_{C^1} \cdots \|F_r\|_{C^1}$.
\end{lem}
\begin{proof}
   Fix $r \geq 1$. For $1\leq i \leq r$, define $\tilde{F}_i: \Mat \rightarrow \R$ as $$\tilde{F}_i(\theta)= F_i(\hat{a}^{l_i} u(\theta)x).$$
   Clearly
   \begin{align}
       \label{eq: 1 1 1 1}
       \|\tilde{F}_i\|_{\Lip} \ll \rho^{-l_i} \|F_i\|_{C^1}.
   \end{align}
   Note that Lipschitz norm of the function \eqref{eq: aaa 1} equals
    \begin{align*}
       & \|F_0 \cdot \tilde{F}_1 \cdots \tilde{F}_r \|_{\Lip} \\
        &\leq \|F_0\|_\Lip \|\tilde{F}_1\|_{C^0} \cdots \|\tilde{F}_r\|_{C^0} +  \|F_0\|_{C^0} \sum_{i=1}^r \|\tilde{F}_i\|_{\Lip} \left( \prod_{j \neq 1} \|\tilde{F}_j\|_{C^0} \right) \\
        &\ll \|F_0\|_{C^1} \|{F}_1\|_{C^0} \cdots \|{F}_r\|_{C^0} +  \|F_0\|_{C^0} \sum_{i=1}^r \rho^{-l_i} \|{F}_i\|_{C^1} \left( \prod_{j \neq 1} \|{F}_j\|_{C^0} \right) \\
        &\ll \rho^{-\max\{l_i\}} \|F_0\|_{C^1} \cdots \|F_r\|_{C^1},
    \end{align*}
    where implied constant depend only on $r$. Hence proved.
\end{proof}

\subsection{Induction}

In this section, we do the induction to prove that Theorem \ref{thm: Effective Multi-equidistribution of fractal measure}, for the case when $t_1, \ldots, t_r$ belongs to the set $\{t \in \R : a_t = \hat{a}^{2l} \text{ for some } l \in \Z \}$. In particular, we prove the following theorem.

\begin{lem}
\label{lem: induction}
    For every $F_0 \in C^\infty(\Mat)$, $F_1, \ldots , F_r \in C_c^\infty(\X)$, $x \in \X$ and $l_1 \leq \cdot \leq l_r$, we have
    \begin{align*}
        \int_{\Mat} F_0(\theta) F_1(\hat{a}^{2l_1} u(\theta) x) \cdots F_r(\hat{a}^{2l_r} u(\theta)x) \, d\mu(\theta) &= \mu(F_0) \mu_\X(F_1) \cdots \mu_\X(F_r) \\
        &+ O_{x,r}(e^{-\delta_0 D(l_1,\ldots , l_r) } \|F_0\|_{C^k} \cdots \|F_r\|_{C^k} ),
    \end{align*}
    where $D(l_1,\ldots , l_r) = \min\{l_i , |l_i-l_j|: 1 \leq i,j \leq r, i \neq j\}$.
\end{lem}
\begin{proof}
    We will use induction to prove the lemma. For $r=0$, the lemma follows trivially. Now assume that $r\geq 1$ and lemma holds for $r-1$. Then using lemma \ref{lem: induction step} for $j=2l_r$, $l= l_r+ l_{r-1}$ for the function $\theta \mapsto F_0(\theta) F_1(\hat{a}^{2l_1} u(\theta) x) \cdots F_{r-1}(\hat{a}^{2l_r} u(\theta)x)$ in place of $F_1$ and the function $F_r$ in place of $F_2$ and then applying lemma \ref{lem: C^k norm of restricted function} to get
    \begin{align}
        &\int_{\Mat} F_0(\theta) F_1(\hat{a}^{2l_1} u(\theta) x) \cdots F_r(\hat{a}^{2l_r} u(\theta)x) \, d\mu(\theta) \nonumber \\
        &=\mu_\X(F_r) \int_{\Mat} F_0(\theta) F_1(\hat{a}^{2l_1} u(\theta) x) \cdots F_{r-1}(\hat{a}^{2l_r} u(\theta)x)\, d\mu(\theta) \nonumber \\
        &+ O_{x,r}(e^{-\delta_0 D(l_1,\ldots , l_r)} \|F_0\|_{C^k} \cdots \|F_r\|_{C^k}) + \rho^{l_r+l_{r-1}} \rho^{-2l_{r-1}}  \|F_0\|_{C^k} \cdots \|F_r\|_{C^k}). \label{eq: 2 2 2 2}
    \end{align}
    Now using \eqref{eq: 2 2 2 2} inductively and noting that $\delta_0 < - \log \rho$, the lemma holds.
\end{proof}

\subsection{Proof of Theorem \ref{thm: Effective Multi-equidistribution of fractal measure}}
 We now derive Theorem \ref{thm: Effective Multi-equidistribution of fractal measure} from Lemma \ref{lem: induction}.
\begin{proof}[Proof of Theorem \ref{thm: Effective Multi-equidistribution of fractal measure}]
    Let $\delta = \delta_0/ 2 \log (\rho^{\frac{-m}{m+n}}) $. Without loss of generality, we may assume that $t_1 \leq  \cdots \leq t_r$. Let $l_i$ be the unique integer satisfying
    $$
    (\rho^{\frac{-m}{m+n}})^{2l_i} \leq e^{t_i} < (\rho^{\frac{-m}{m+n}})^{2(l_i+1)},
    $$
    that is $l_i$ is the greatest integer less than or equal to $ \frac{t_i}{2 \log (\rho^{\frac{-m}{m+n}})}$ and let $s_i = t_i - {2l_i}\log (\rho^{\frac{-m}{m+n}}) \in [0, 2\log (\rho^{\frac{-m}{m+n}}) ) $. Then we have
    \begin{align*}
         &\int_{\Mat}  F_0(\theta) F_1(a_{t_1} u(\theta) x) \cdots F_r(a_{t_r}u(\theta)x) \, d\mu(\theta) \\
         &= \int_{\Mat}  F_0(\theta) F_1 \circ a_{s_1}(\hat{a}^{2 l_1} u(\theta) x) \cdots F_r\circ a_{s_r } (\hat{a}^{2l_r} u(\theta)x) \, d\mu(\theta) \\
         &= \mu(F_0) \mu_\X(F_1 \circ a_{s_1}) \cdots \mu_\X(F_r \circ a_{s_r}) + O_{x,r}(e^{-\delta_0 D(l_1,\ldots , l_r)} \|F_0\|_{C^k} \|F_1 \circ a_{s_1}\|_{C^k} \cdots \|F_r \circ a_{s_r}\|_{C^k} ) \\
         &=  \mu(F_0) \mu_\X(F_1 ) \cdots \mu_\X(F_r) + O_{x,r}(e^{-\delta D(t_1,\ldots , t_r)} \|F_0\|_{C^k} \|F_1 \|_{C^k} \cdots \|F_r \|_{C^k} ).
    \end{align*}
    Hence proved.
\end{proof}

\section{Notation}

Throughout the paper, we fix norms $\|\cdot\|_m$ on $\mathbb{R}^m$ and $\|\cdot\|_n$ on $\mathbb{R}^n$. By rescaling these norms if necessary, we may assume that the unit balls
\[
\left\{ x \in \mathbb{R}^m : \|x\|_m \leq 1 \right\} \quad \text{and} \quad \left\{ y \in \mathbb{R}^n : \|y\|_n \leq 1 \right\}
\]
have Lebesgue measure greater than $2^m$ and $2^n$, respectively. For notational convenience, we will use the same symbol $\|\cdot\|$ to denote both $\|\cdot\|_m$ and $\|\cdot\|_n$.

Let $\pi_1: \mathbb{R}^{m+n} = \mathbb{R}^m \times \mathbb{R}^n \to \mathbb{R}^m$ and $\pi_2: \mathbb{R}^{m+n} \to \mathbb{R}^n$ denote the natural projections onto the first and second components, respectively. We equip $\mathbb{R}^{m+n}$ with the norm
\[
\|v\| := \max\left\{ \|\pi_1(v)\|,\ \|\pi_2(v)\| \right\}, \quad v \in \mathbb{R}^{m+n}.
\]

Given $v \in \mathbb{R}^{m+n}$, we define the associated box
\begin{align}
    \label{eq: def C_v}
    C_v := \left\{ (x, y) \in \mathbb{R}^m \times \mathbb{R}^n : \|x\| \leq \|\pi_1(v)\|,\ \|y\| \leq \|\pi_2(v)\| \right\}.
\end{align}

For $\Lambda \in \X$, define the set
\[
S_\Lambda := \left\{ v \in \Lambda : 
\begin{array}{l}
    \#\left( \Lambda_\prim \cap C_v \right) = \{ \pm v \}, \\[2pt]
    1 \leq \|\pi_2(v)\| < e, \\[2pt]
    \|\pi_1(v)\| \leq 1
\end{array}
\right\}.
\]
Note that condition $\|\pi_1(v)\| \leq 1$ is, in fact, redundant: if $v \in \Lambda \in \X$ satisfies $1 \leq \|\pi_2(v)\| < e$ and $\#( \Lambda_\prim \cap C_v ) = 0$, then Minkowski’s convex body theorem implies that $\|\pi_1(v)\| \leq 1$.

Moreover, by \cite[Lem.~4.1]{AG25}, we have
\begin{align}
    \label{eq: def M}
    M := \sup \left\{ \#S_\Lambda : \Lambda \in \X \right\} < \infty,
\end{align}
i.e., there exists a constant $M \in \mathbb{N}$ such that $\#S_\Lambda \leq M$ for all $\Lambda \in \X$. We fix such an $M$ for the remainder of the paper.

\section{The function \texorpdfstring{$f$}{f}}

For the remainder of the paper, we fix a function
\[
F : \R_{\geq 0} \times \Sphere^m \times \Sphere^n \to \R.
\]

Define the map $\phi : \R^m \times \R^n \to \R_{\geq 0} \times \Sphere^m \times \Sphere^n$ by
\[
\phi(x, y) := \left( \|x\|^m \|y\|^n,\ \frac{x}{\|x\|},\ \frac{y}{\|y\|} \right).
\]

We now define the function $f : \X \to \R$ by
\begin{align}
    \label{eq: def f}
    f(\Lambda) := \sum_{v \in S_\Lambda} (F \circ \phi)(v).
\end{align}

In addition, for any $\theta \in \Mat$ and real numbers $0 \leq a < b$, define
\[
\cF_F(\theta, a, b) := \sum_{\substack{(p,q) \text{ is a best approximate} \\ e^a \leq \|q\| < e^b}} F\left( \|p + \theta q\|^m \|q\|^n,\ \frac{p + \theta q}{\|p + \theta q\|},\ \frac{q}{\|q\|} \right).
\]
Note that this generalizes the earlier definition \eqref{eq: def cF}, with
\[
\cF_F(\theta, T) = \cF_F(\theta, 0, T).
\]
The following lemma illustrates the role of the function $f$ in connecting the dynamics on the space $\X$ with the Diophantine sums $\cF_F$.

\begin{lem}
\label{lem: correpondence with best approx}
For all $T \geq 0$, we have
\begin{align}
    \sum_{i=1}^{\lfloor T \rfloor} f(a_i u(\theta)\Gamma) \leq \cF_F(\theta, T) \leq \sum_{i=1}^{\lfloor T \rfloor + 1} f(a_i u(\theta)\Gamma),
\end{align}
where $\lfloor T \rfloor$ denotes the integer part of $T$.
\end{lem}

\begin{proof}
It suffices to show that for every $l \in \Z_{\geq 0}$, we have
\[
\cF_F(\theta, l, l+1) = f(a_l u(\theta)\Gamma).
\]
To see this, recall from the proof of Lemma~3.1 in \cite{AG25} that
\begin{align}
    \label{eq: aa 1}
    S_{a_l u(\theta)\Gamma} = \left\{ a_l(p + \theta q, q) : (p,q) \in \text{Best}(\theta; l, l+1) \right\},
\end{align}
where $\text{Best}(\theta; l, l+1)$ denotes the set of best approximations $(p,q)$ of $\theta$ satisfying $e^l \leq \|q\| < e^{l+1}$.

Now observe that for any such $(p,q)$,
\begin{align}
    \label{eq: aa 2}
    F\left( \|p + \theta q\|^m \|q\|^n,\ \frac{p + \theta q}{\|p + \theta q\|},\ \frac{q}{\|q\|} \right)
    = \left(F \circ \phi\right)\left( a_l(p + \theta q, q) \right).
\end{align}
Combining~\eqref{eq: aa 1} and~\eqref{eq: aa 2}, and using the definition of $f$, we obtain
\[
\cF_F(\theta, l, l+1) = \sum_{v \in S_{a_l u(\theta)\Gamma}} (F \circ \phi)(v) = f(a_l u(\theta)\Gamma),
\]
as desired. Hence proved.
\end{proof}

We now show that the function $f$ is bounded.

\begin{lem}
\label{lem: bounded f}
For all $\Lambda \in \X$, we have
\[
f(\Lambda) \leq M \|F\|_{C^0}.
\]
\end{lem}

\begin{proof}
This follows directly from the definition of $f$ in \eqref{eq: def f}, together with the uniform bound $\#S_\Lambda \leq M$ established in \eqref{eq: def M}.
\end{proof}

\vspace{1em}

Next, we study the continuity properties of the function $f$. This will be a key ingredient in applying Theorem~\ref{thm: main abstract theorem} later on.

\begin{lem}
\label{lem: perturbed difference}
There exists a constant $C > 0$ such that for all small enough $\e > 0$ and all $g$ in an $\e$-ball around the identity in $G$, we have
\begin{align}
\label{eq: lem: perturbed difference 1}
|f(g\Lambda)- f(\Lambda)| \leq \left( \sum_{v\in \Lambda_\prim } \varphi_{C\e}(v) + \sum_{\substack{v,w \in \Lambda_\prim \\ w \neq \pm v}} \Phi_{C\e}(v,w) \right) \|F\|_{C^0} + M \|F \circ \phi\|_{C^1} \cdot C \e,
\end{align}
where for each $\delta > 0$, the auxiliary functions $\varphi_\delta$ and $\Phi_\delta$ are defined by
\begin{align}
\label{eq: def varphi e}
\varphi_\delta(v) :=
\begin{cases}
1 & \text{if } v \in \{ (x, y) : \|x\| \leq 1+\delta,\ 1-\delta \leq \|y\| \leq e+\delta \} \\
&\quad \text{and } v \notin \{ (x, y) : \|x\| < 1-\delta,\ 1+\delta < \|y\| < e-\delta \}, \\
0 & \text{otherwise},
\end{cases}
\end{align}
and
\begin{align}
\label{eq: def Phi e}
\Phi_\delta(v, w) :=
\begin{cases}
1 & \text{if } v, w \in \{ (x, y) : \|x\| \leq 1+\delta,\ \|y\| \leq e+\delta \}, \\
&\quad \text{and either } |\|\pi_1(v)\| - \|\pi_1(w)\|| \leq \delta\ \text{or}\ |\|\pi_2(v)\| - \|\pi_2(w)\|| \leq \delta, \\
0 & \text{otherwise}.
\end{cases}
\end{align}

    \end{lem}
    \begin{proof}
Fix $C_0 > 0$ such that for all sufficiently small $\e > 0$, and all $g$ in the $\e$-ball around the identity in $G$, we have
\begin{align}
    \label{eq: w 1}
    \|gv - v\| \leq C_0 \e \|v\|, \quad \text{for all } v \in \R^{m+n}.
\end{align}
Set $C := 2 C_0 e$. Fix $\e > 0$ and let $g$ lie in the $\e$-ball around the identity in $G$.

By the proof of \cite[Lem. 4.2]{AG25}, we have the following estimate:
\begin{align}
    \label{eq: bb 1}
    \#(S_\Lambda \setminus g^{-1}S_{g\Lambda}) + \#(g^{-1}S_{g\Lambda} \setminus S_{\Lambda}) \leq \sum_{v \in \Lambda_\prim} \varphi_{C\e}(v) + \sum_{\substack{v, w \in \Lambda_\prim \\ w \neq \pm v}} \Phi_{C\e}(v, w).
\end{align}

Next, observe that for all $v \in S_\Lambda \cap g^{-1}S_{g\Lambda}$, we have $\|v\| \leq e$, and so by \eqref{eq: w 1},
\[
\|gv - v\| \leq C_0 \e e \leq C\e.
\]
Hence,
\begin{align}
    \label{eq: bb 2}
    |F \circ \phi(gv) - F \circ \phi(v)| \leq \|F \circ \phi\|_{\mathrm{Lip}} \cdot \|gv - v\| \leq C \|F \circ \phi\|_{C^1} \e.
\end{align}

To prove the lemma, we write:
\begin{align*}
    &|f(g\Lambda) - f(\Lambda)| 
    = \left| \sum_{w \in S_{g\Lambda}} F \circ \phi(w) - \sum_{v \in S_\Lambda} F \circ \phi(v) \right| \\
    &= \left| \sum_{w \in S_{g\Lambda} \cap gS_\Lambda} F \circ \phi(w) - \sum_{v \in S_\Lambda \cap g^{-1}S_{g\Lambda}} F \circ \phi(v) \right|  + \left| \sum_{w \in S_{g\Lambda} \setminus gS_\Lambda} F \circ \phi(w) \right| + \left| \sum_{v \in S_\Lambda \setminus g^{-1}S_{g\Lambda}} F \circ \phi(v) \right| \\
    &\leq \sum_{v \in S_\Lambda \cap g^{-1}S_{g\Lambda}} |F \circ \phi(gv) - F \circ \phi(v)|  + \left( \#(S_{g\Lambda} \setminus gS_\Lambda) + \#(S_\Lambda \setminus g^{-1}S_{g\Lambda}) \right) \cdot \|F \circ \phi\|_{C^0}.
\end{align*}

Applying \eqref{eq: def M}, \eqref{eq: bb 1}, and \eqref{eq: bb 2}, we obtain
\[
|f(g\Lambda) - f(\Lambda)| \leq M \cdot C \|F \circ \phi\|_{C^1} \e + \left( \sum_{v \in \Lambda_\prim} \varphi_{C\e}(v) + \sum_{\substack{v, w \in \Lambda_\prim \\ w \neq \pm v}} \Phi_{C\e}(v, w) \right) \|F \circ \phi\|_{C^0}.
\]

This completes the proof.
\end{proof}


\section{Proof of Theorem \ref{thm: main theorem general}}

To prove Theorem~\ref{thm: main theorem general}, we will apply the following general result from~\cite[Thm.~5.1]{AG25}.

\begin{thm}[{\cite[Thm. 5.1]{AG25}}]
\label{thm: main abstract theorem}
    Let $f$ be a bounded function on $\X$. Assume that there exists a family of non-negative measurable functions $\{\tau_\e\}_\e$ on $\X$ satisfying the following.
    \begin{itemize}
        \item There exists a constant $C$ such that 
        \begin{align}
            \label{eq: con: xi integral}
            \int_\X \tau_\e \, d\mu_\X \leq C \e
        \end{align}
        \item For all small enough $\e>0$, and $g$ in $\e$-neighbourhood of identity in $G$, we have
    \begin{align}
        \label{eq:con: f cont}
        |f(g\Lambda)-f(\Lambda)| \leq \tau_\e(\Lambda) \quad \text{ for all } \Lambda \in \X.
    \end{align}
    \end{itemize}
    Then for any measure $\mu$ on $\Mat$ satisfying condition (EMEI), the following holds.
    \begin{enumerate}
        \item[(i)] For any $\varepsilon > 0$ and for $\mu$-almost every $\theta \in \Mat$, we have
        \begin{align*}
            \sum_{s=0}^{N-1} f(a_{s}u(\theta)\Gamma) = N \gamma  + o(N^{1/2} \log^{\tfrac{1}{2}+\e} N),
        \end{align*}
        where
        $$
        \gamma= \mu_{\X}(f).
        $$
        \item[(ii)] For $N \geq 1$, let us define $F_N: \X \rightarrow \R$ as
    \begin{equation*}
        F_N(\Lambda) =\frac{1}{\sqrt{N}} \sum_{i=0}^{N-1} \left( f\circ a_i(\Lambda) - \gamma \right).
    \end{equation*}
    Then for every $x \in \R$, we have
    \begin{align*}
         \mu(\{\theta: F_N(u(\theta)\Gamma) < x \}) \rightarrow \mathrm{Norm}_{\sigma}(x) 
    \end{align*}
    as $N \rightarrow \infty$, where
    $$
    \sigma^2 =\sum_{s \in \Z} \left(  \int_{\X} f(a_s\Lambda) f(\Lambda)\, d\mu_\X(\Lambda) - \mu_{\X}(f)^2 \right).
    $$
    \end{enumerate}
\end{thm}
\vspace{0.2in}

We now proceed to verify that the function \( f \) satisfy the assumptions of Theorem~\ref{thm: main abstract theorem}, thereby completing the proof of Theorem~\ref{thm: main theorem general}.

\begin{proof}[Proof of Theorem~\ref{thm: main theorem general}]
We begin by applying Siegel's mean value theorem \cite{sie}, which yields
\begin{align}
    \label{eq: int varphi}
    \int_\X \sum_{v \in \Lambda_\prim} \varphi_\e(v) \, d\mu_\X(\Lambda)
    = \frac{1}{\zeta(m+n)} \int_{\mathbb{R}^{m+n}} \varphi_\e(v) \, dv
    \ll \e,
\end{align}
where the implied constant is independent of \( \e \).

Next, applying both the Siegel mean value theorem and Rogers' formula~\cite[Thm.~5]{rog}, we obtain
\begin{align}
    \label{eq: int PHI}
    \int_\X \sum_{\substack{v, w \in \Lambda_\prim \\ w \neq \pm v}} \Phi_\e(v, w) \, d\mu_\X(\Lambda)
    &= \int_\X \sum_{v, w \in \Lambda_\prim} \Phi_\e(v, w) \, d\mu_\X(\Lambda) 
    - \int_\X \sum_{v \in \Lambda_\prim} \left( \Phi_\e(v, v) + \Phi_\e(v, -v) \right) \, d\mu_\X(\Lambda) \nonumber \\
    &= \frac{1}{\zeta(m+n)^2} \int_{\mathbb{R}^{m+n} \times \mathbb{R}^{m+n}} \Phi_\e(v, w) \, dv\, dw 
    \ll \e,
\end{align}
again with implied constants uniform in \( \e \).

The two estimates \eqref{eq: int varphi} and \eqref{eq: int PHI} ensure that the integral condition~\eqref{eq: con: xi integral} holds for the error function \( \tau_\e \) given by Lemma~\ref{lem: perturbed difference}. Moreover, Lemma~\ref{lem: bounded f} guarantees that \( f \) is bounded on \( \X \), and Lemma~\ref{lem: perturbed difference} verifies the continuity condition~\eqref{eq:con: f cont}.

Thus, all hypotheses of Theorem~\ref{thm: main abstract theorem} are satisfied. The conclusion of Theorem~\ref{thm: main theorem general} now follows directly.
\end{proof}


\bibliography{Biblio2}
\end{document}